\documentclass[11pt, oneside]{article}   	
\usepackage{geometry}                		
\usepackage{graphicx}				
\usepackage{mathrsfs,amssymb,amsmath,amsthm, color,stmaryrd}

\newtheorem{thm}{Theorem}[section]

\newtheorem{dfn}[thm]{Definition}
\newtheorem{lmm}[thm]{Lemma}
\newtheorem{prp}[thm]{Proposition}

\newtheorem{rem}[thm]{Remark}
\numberwithin{equation}{section}

\newcommand{\bbB}{\mathbb{B}}

\newcommand{\bbN}{\mathbb{N}}
\newcommand{\bbR}{\mathbb{R}}

\newcommand{\pl}{\mathbin{\varolessthan}}
\newcommand{\mpl}[1]{\mathbin{{}^{#1}\varolessthan}}
\newcommand{\rs}{\varodot}

\newcommand{\bsf}[1]{\boldsymbol{f}\!_{#1}}

\allowdisplaybreaks

\title{A note on the Taylor estimates of iterated paraproducts}
\author{Masato Hoshino\footnote{
Graduate School of Engineering Science, Osaka University, 
1-3, Machikaneyama, Toyonaka, Osaka, 560-8531, Japan. 
Email: {\tt hoshino@sigmath.es.osaka-u.ac.jp}
}
}

\date{}							

\begin{document}
\maketitle

\begin{abstract}      
Bony's paraproduct is one of the main tools in the theory of paracontrolled calculus.
The paraproduct is usually defined via Fourier analysis, so it is not a local operator.
In the previous researches \cite{Hos20, Hos21}, however, the author proved that the pointwise estimate like \eqref{1:eq:Taylorf<g} holds for the paraproduct and its iterated versions when the sum of the regularities is smaller than $1$.
The aim of this article is to extend these results for higher regularities.
\end{abstract}

\section{Introduction}

This is a short note on the extensions of \cite{Hos20, Hos21}.
Besov spaces on Euclidean space $\bbR^d$ are usually defined via Fourier analysis,
and also equivalently defined via estimates of Taylor remainders.
Indeed, if $\alpha\in(0,\infty)\setminus\bbN$ and $f\in B_{\infty,\infty}^\alpha=B_{\infty,\infty}^\alpha(\bbR^d)$, then $f$ is differentiable up to order $[\alpha]$ (integer part of $\alpha$) and has the estimate
\begin{align}\label{1:def:Omegaalpha}
\Omega^\alpha(f)(\cdot,x):=
f(\cdot)-\sum_{|k|<\alpha}\frac{(\cdot-x)^k}{k!}\partial^kf(x)=O(|\cdot-x|^\alpha),\qquad x\in\bbR^d.
\end{align}
See \cite[Theorem 2.36]{BCD11} for instance.
In \cite{Hos20, Hos21}, the author proved similar estimates for \emph{Bony's paraproduct}.
One of the results of \cite{Hos20} says that, if $\alpha,\beta>0$, $\alpha+\beta<1$, $f\in B_{\infty,\infty}^\alpha$, and $g\in B_{\infty,\infty}^\beta$, then the paraproduct $f\pl g$ (see Proposition \ref{prp:paraproduct}) has an estimate
\begin{align}\label{1:eq:Taylorf<g}
(f\pl g)(\cdot)=(f\pl g)(x)-f(x)(g(\cdot)-g(x))+O(|\cdot-x|^{\alpha+\beta}),\qquad x\in\bbR^d.
\end{align}
In \cite{Hos20, Hos21}, similar estimates were obtained for iterated paraproducts $((f_1\pl f_2)\pl\cdots)\pl f_n$,
and applied to the proof of generalized commutator estimates.
The commutator estimate first proved in \cite[Lemma 2.4]{GIP15} has an important role in the theory of \emph{paracontrolled calculus} \cite{GIP15, BB19}.
See the remarks after Theorem \ref{mainthm} for an overview of the theory.
In these articles, however, the sum of regularities is always assumed to be smaller than $1$.
The aim of this article is to extend these results for higher regularities. 

Fix $d\ge1$ throughout this article.
For any multiindex $k=(k_i)_{i=1}^d\in\bbN^d$, denote $|k|:=\sum_{i=1}^dk_i$
and $k!=\prod_{i=1}^dk_i!$.
We use the notation $a\lesssim b$ to mean $a\le cb$ for some constant $c>0$.

\begin{thm}\label{mainthm}
Let $\alpha,\beta,\gamma\in(0,\infty)\setminus\bbN$, $\alpha+\beta,\beta+\gamma,\alpha+\beta+\gamma\notin\bbN$, $f\in B_{\infty,\infty}^\alpha$, $g\in B_{\infty,\infty}^\beta$, and $h\in B_{\infty,\infty}^\gamma$.
\begin{enumerate}
\item[\rm (1)]
Define
\begin{align}\label{mainthm1}
\begin{aligned}
\Omega^{\alpha,\beta}(f,g)(y,x)&:=(f\pl g)(y)-\sum_{|k|<\alpha+\beta}\frac{(y-x)^k}{k!}D^k(f,g)(x)\\
&\quad-\sum_{|k|<\alpha}\frac{(y-x)^k}{k!}\partial^kf(x)\Omega^\beta(g)(y,x),
\end{aligned}
\end{align}
where
\begin{align}\label{eq:D(f,g)}
D^k(f,g)=\partial^k(f\pl g)-\sum_{\substack{k_1+k_2=k\\|k_1|<\alpha,\ |k_2|\ge\beta}}\binom{k}{k_1}\partial^{k_1}f\partial^{k_2}g
\end{align}
is a bounded continuous function (see Lemma \ref{lem:D^kisBC}), and $\Omega^\beta(g)$ is the same notation as \eqref{1:def:Omegaalpha}.
Then 
$$
\sup_{x,y\in\bbR^d}\frac{|\Omega^{\alpha,\beta}(f,g)(y,x)|}{|y-x|^{\alpha+\beta}}\lesssim\|f\|_{B_{\infty,\infty}^\alpha}\|g\|_{B_{\infty,\infty}^\beta}.
$$
\item[\rm (2)]
Define
\begin{align}\label{mainthm2}
\begin{aligned}
\Omega^{\alpha,\beta,\gamma}(f,g,h)(y,x)
&:=\big((f\pl g)\pl h\big)(y)
-\sum_{|k|<\alpha+\beta+\gamma}\frac{(y-x)^k}{k!}D^k(f,g,h)(x)\\
&\quad-\sum_{|k+\ell|<\alpha}\frac{(y-x)^k}{k!}\partial^{k+\ell}f(x)\Omega_\ell^{\beta,\gamma}(g,h)(y,x)\\
&\quad-\sum_{|k|<\alpha+\beta}\frac{(y-x)^k}{k!}D^k(f,g)(x)\Omega^\gamma(h)(y,x)
\end{aligned}
\end{align}
for some choice of bounded continuous functions $D^k(f,g,h)$,
and
\begin{align*}
\Omega_\ell^{\beta,\gamma}(g,h)
=
\begin{cases}
\Omega^{\beta,\gamma}(g,h),&\ell=0,\\
\Omega^{\beta+\gamma+|\ell|}(g\mpl{\ell}h),&\ell\neq0
\end{cases}
\end{align*}
for some $g\mpl{\ell}h\in B_{\infty,\infty}^{\beta+\gamma+|\ell|}$.
Here $\Omega^{\beta,\gamma}$ and $\Omega^\gamma$ are the same notations as \eqref{mainthm1} and \eqref{1:def:Omegaalpha}, respectively.
Then
$$
\sup_{x,y\in\bbR^d}\frac{|\Omega^{\alpha,\beta}(f,g,h)(y,x)|}{|y-x|^{\alpha+\beta+\gamma}}\lesssim\|f\|_{B_{\infty,\infty}^\alpha}\|g\|_{B_{\infty,\infty}^\beta}\|h\|_{B_{\infty,\infty}^\gamma}.
$$
\end{enumerate}
\end{thm}

\begin{rem}
In \cite{Hos20, Hos21}, the author proved the pointwise estimates for iterated paraproducts $((f_1\pl f_2)\pl\cdots)\pl f_n$.
For $n=2,3$, this result takes the form of \eqref{1:eq:Taylorf<g} and
\begin{align*}
&\big((f\pl g)\pl h\big)(y)-\big((f\pl g)\pl h\big)(x)\\
&\quad-f(x)\big\{(g\pl h)(y)-(g\pl h)(x)\big\}
-(f\pl g)(x)\big(h(y)-h(x)\big)\\
&=O(|y-x|^{\alpha+\beta+\gamma})
\end{align*}
for $(f,g,h)\in B_{\infty,\infty}^\alpha\times B_{\infty,\infty}^\beta\times B_{\infty,\infty}^\gamma$ with $\alpha,\beta,\gamma>0$ and $\alpha+\beta+\gamma<1$.
Here is a restriction that the sum of regularities should be smaller than $1$.
The formulas \eqref{mainthm1} and \eqref{mainthm2} are extensions of \cite{Hos20, Hos21} for higher regularities.
\end{rem}

Let us briefly recall the role of paraproducts and commutators in the theory of paracontrolled calculus.
The main subject of paracontrolled calculus is solving PDEs with stochastic noise, for example, (generalized) parabolic Anderson model
\begin{align}\label{gpam}
(\partial_t-\Delta)u=f(u)\xi,
\end{align}
where $\xi$ is a stochastic noise. A typical choice of $\xi$ is a Gaussian white noise on $\bbR^d$, so it has a regularity only less than $-\frac{d}2$.
Then $u$ would have a regularity less than $2-\frac{d}2$ by Schauder estimate, but it is not enough to define the product $f(u)\xi$, since the sum of regularities of $f(u)$ and $\xi$ is $(2-\frac{d}2)+(-\frac{d}2)\le0$ if $d\ge2$.
To overcome this difficulty, we impose some structural assumption on $u$. In \cite{GIP15}, the authors solved \eqref{gpam} in $d=2$ by assuming that $u$ is of the form
$$
u=u'\pl Z+u^\sharp,\qquad u'\in B_{\infty,\infty}^\alpha,\ u^\sharp\in B_{\infty,\infty}^{2\alpha}
$$
with some $\alpha<1$.
In \cite{BB19}, the authors solved \eqref{gpam} in $d=3$ by assuming
\begin{align*}
u&=\sum_{i=1}^3u_i\pl Z_i+u^\sharp,\qquad u^\sharp\in B_{\infty,\infty}^{4\alpha},\\
u_i&=\sum_{1\le i+j\le 3}u_{ij}\pl Z_j+u_i^\sharp,\qquad u_i^\sharp\in B_{\infty,\infty}^{(4-i)\alpha},\\
u_{ij}&=\sum_{1\le i+j+k\le 3}u_{ijk}\pl Z_k+u_{ij}^\sharp,\qquad u_{ij}^\sharp\in B_{\infty,\infty}^{(4-i-j)\alpha},\ u_{ijk}\in B_{\infty,\infty}^{\alpha}
\end{align*}
over $1\le i,i+j,i+j+k\le 3$, with some $\alpha<\frac12$.
Here $Z$ and $Z_i$'s are some explicit functionals of $\xi$, defined by a stochastic approach called renormalization.
For the product of unknown functions and $\xi$, we use the commutator estimates.
In 2d case for instance, the trilinear operator
$$
C(u',Z,\xi):=(u'\pl Z)\rs\xi-u'(Z\rs\xi)
$$
is called a commutator, where $\rs$ is a resonant product (see Proposition \ref{prp:paraproduct}).
Two resonants in the right-hand side are ill-defined since $Z\in B_{\infty,\infty}^\alpha$ and $\xi\in B_{\infty,\infty}^{\alpha-2}$ with $\alpha<1$,
but it turns out that $C$ is actually a continuous operator of $(u',Z,\xi)$ if $\alpha\in(\frac23,1)$.
It yields that $(u'\pl Z)\rs\xi:=u'(Z\rs\xi)+C(u',Z,\xi)$ is well-defined and continuous from $(u',Z,\xi,Z\rs\xi)$.
We can extend this argument to define the resonant $f(u)\rs\xi$, by using a paralinearization formula \cite[Theorem 2.92]{BCD11}.
In 3d case, more commutations between iterated paraproducts and resonants are needed. See \cite{BB19} for details.
%
%

Instead of paracontrolled calculus, we can also prove the same well-posedness results by the theory of \emph{regularity structures} \cite{Hai14}.
One of the differences between those theories is how to write the assumptions of solutions.
In the former, main tools are paraproducts and commutators, so it is purely a harmonic analysis.
In the latter, the solution is defined as a ``modelled distribution," which is a kind of generalized Taylor expansion like \eqref{mainthm1} and \eqref{mainthm2}.
Hence, it is a natural question whether the Fourier picture and the Taylor picture are equivalent.
Theorem \ref{mainthm} is an answer to that question in the particular case.
For other studies on this direction, see Section 6 of \cite{GIP15}, and \cite{MP20, BH21a, BH21b}.

This article is organized as follows.
Section \ref{section:preliminaries} is devoted to the proof of Theorem \ref{thm:abstractparaproduct}, which is an essential part of this article.
In Section \ref{section:pointwise}, we prove Theorem \ref{mainthm} as a corollary of Theorem \ref{thm:abstractparaproduct}.

We remark that the formulas in Theorem \ref{thm:abstractparaproduct} are simpler than those in Theorem \ref{mainthm}. In the former, $\Omega_\ell^{\beta,\gamma}(g,h)$ term with $\ell\neq0$ does not appear.
It may imply that, the multicomponent operators as in Definition \ref{dfn:simpleparaproduct} instead of the iterated paraproducts $((f_1\pl f_2)\pl\cdots)\pl f_n$ make the things simpler, but we do not touch such a problem in this article.


\section{Key estimates}\label{section:preliminaries}

The aim of this section is to prove Theorem \ref{thm:abstractparaproduct}.
First we recall the basics of Littlewood-Paley theory and Bony's paraproduct.

\begin{dfn}\label{dfn:paraproduct}
Fix smooth, nonnegative, and radial functions $\chi$ and $\rho$ on $\mathbb{R}^d$ such that
\begin{itemize}
\item $\mathop{\rm supp}(\chi)\subset\{x\ ;\, |x|<\frac43\}$ and $\mathop{\rm supp}(\rho)\subset\{x\ ;\, \frac34<|x|<\frac83\}$,
\item $\chi(x)+\sum_{j=0}^\infty\rho(2^{-j}x)=1$ for all $x\in\mathbb{R}^d$.
\end{itemize}
Set $\rho_{-1}:=\chi$ and $\rho_j=\rho(2^{-j}\cdot)$ for $j\ge0$.
For each $j\ge-1$, define the Littlewood-Paley blocks
$$
\Delta_jf:=\mathcal{F}^{-1}(\rho_j\mathcal{F}f)
$$
of $f\in\mathcal{S}'=\mathcal{S}'(\mathbb{R}^d)$, where $\mathcal{F}$ and $\mathcal{F}^{-1}$ are Fourier transform and its inverse, respectively.
Moreover, for any $\alpha\in\mathbb{R}$ and $p,q\in[1,\infty]$, define $B_{p,q}^\alpha$ as the set of all $f\in\mathcal{S}'$ such that
$$
\|f\|_{B_{p,q}^\alpha}:=\Big\|\big\{2^{j\alpha}\|\Delta_jf\|_{L^p(\mathbb{R}^d)}\big\}_{j=-1}^\infty\Big\|_{\ell^q}<\infty.
$$
\end{dfn}

Littlewood-Paley blocks decompose $f\in\mathcal{S}'(\mathbb{R}^d)$ into the infinite sum of smooth functions $\Delta_jf$.
The support conditions of $\chi$ and $\rho$ are to ensure that spectral supports of $\Delta_if$ and $\Delta_jf$ are disjoint if $|i-j|\ge2$.
One of the applications of this decomposition is Bony's paraproduct.

\begin{prp}[{\cite[Theorems 2.82 and 2.85]{BCD11}}]\label{prp:paraproduct}
For $f,g\in\mathcal{S}'$, define
$$
f\pl g=\sum_{i<j-1}\Delta_if\Delta_jg,\qquad
f\rs g=\sum_{|i-j|\le1} \Delta_if\Delta_jg,
$$
if the right-hand side converges in $\mathcal{S}'$.
Let $\alpha,\beta\in\bbR$ and $p,p_1,p_2,q,q_1,q_2\in[1,\infty]$ be such that $\frac1p=\frac1{p_1}+\frac1{p_2}$ and $\frac1q=\frac1{q_1}+\frac1{q_2}$.
\begin{enumerate}
\item[\rm (1)]
If $\alpha\neq0$, then $\|f\pl g\|_{B_{p,q}^{\alpha\wedge 0+\beta}}\lesssim\|f\|_{B_{p_1,q_1}^\alpha}\|g\|_{B_{p_2,q_2}^\beta}$.
\item[\rm (2)]
If $\alpha+\beta>0$, then $\|f\rs g\|_{B_{p,q}^{\alpha+\beta}}\lesssim\|f\|_{B_{p_1,q_1}^\alpha}\|g\|_{B_{p_2,q_2}^\beta}$.
\end{enumerate}
\end{prp}

We also consider a slightly more abstract setting.

\begin{dfn}\label{dfn:simpleparaproduct}
Let $\alpha\in\bbR$ and $p,q\in[1,\infty]$.
Define $\bbB_{p,q}^\alpha$ as the set of all sequences $\boldsymbol{f}=\{f^j\}_{j=0}^\infty$ of smooth functions on $\bbR^d$ such that,
$\mathcal{F}f^j$ is supported in $2^jB=\{2^jx\ ;\, x\in B\}$ with some ball $B\subset\bbR^d$ for all $j\ge0$, and
$$
\|\boldsymbol{f}\|_{\bbB_{p,q}^\alpha}:=\Big\|\big\{2^{j\alpha}\|f^j\|_{L^p(\mathbb{R}^d)}\big\}_{j=0}^\infty\Big\|_{\ell^q}<\infty.
$$
\end{dfn}

We treat only $p=q=\infty$ in this article, but the following arguments still work for general $p,q$ by easy modifications.
We write $\bbB^\alpha:=\bbB_{\infty,\infty}^\alpha$ and $\|\cdot\|_\alpha:=\|\cdot\|_{\bbB_{\infty,\infty}^\alpha}$ in short.
Moreover, we treat only $\alpha>0$.
The following estimates follow from the definitions and Bernstein's lemma \cite[Lemma 2.1]{BCD11}.
\begin{itemize}
\item
For any $\boldsymbol{f}\in\bbB^\alpha$ with $\alpha>0$, the series $\sum\boldsymbol{f}:=\sum_{j=0}^\infty f^j$ converges in $B_{\infty,\infty}^\alpha$,
and one has $\|\sum\boldsymbol{f}\|_{B_{\infty,\infty}^\alpha}\lesssim\|\boldsymbol{f}\|_\alpha$.
\item
$\|\partial^k\boldsymbol{f}\|_{\alpha-|k|}\lesssim\|\boldsymbol{f}\|_\alpha$ for any $k\in\bbN^d$, where $\partial^k\boldsymbol{f}:=\{\partial^kf^j\}_{j=0}^\infty$.
If $\alpha>0$, $\sum\partial^k\boldsymbol{f}$ converges to $\partial^k\sum\boldsymbol{f}$ in $\mathcal{S}'$, since
for any $\varphi\in\mathcal{S}$,
\begin{align*}
\Big(\sum_{j=0}^N\partial^kf^j,\varphi\Big)
=(-1)^{|k|}\Big(\sum_{j=0}^Nf^j,\partial^k\varphi\Big)
\xrightarrow{N\to\infty}(-1)^{|k|}\big(\sum\boldsymbol{f},\partial^k\varphi\big)=\big(\partial^k\sum\boldsymbol{f},\varphi\big).
\end{align*}
\end{itemize}

We consider an analogue of Bony's paraproduct for these sequences.
For any $\boldsymbol{f}\in\bbB^\alpha$ and $\boldsymbol{g}\in\bbB^\beta$ with $\alpha,\beta>0$, we define
\begin{align}\label{eq:abstractparaproduct2}
(\boldsymbol{f},\boldsymbol{g}):=\{f^{<j}g^j\}_{j=0}^\infty,
\end{align}
where $f^{<j}:=\sum_{i<j}f^i$.
Since $\|f^{<j}\|_{L^\infty(\mathbb{R}^d)}$ is bounded over $j$, it follows that $(\boldsymbol{f},\boldsymbol{g})\in\bbB^{\beta}$.
Similarly, for any sequence $\bsf{n}\in\bbB^{\alpha_n}$ with $\alpha_n>0$ ($n=1,2,\dots$), we recursively define
$$
(\bsf{1},\bsf{2},\dots,\bsf{n}):=((\bsf{1},\bsf{2},\dots,\bsf{n-1}),\bsf{n}).
$$

\begin{rem}\label{rem:<j-N}
The following arguments still work if we replace $f^{<j}$ with $f^{<j-N}$ for some fixed $N>0$.
For instance, $N=1$ was chosen in the above definition of Bony's paraproduct.
If we define $\bsf{n}=\{\Delta_{j-1}f_n\}_{j=0}^\infty$ for some $f_n\in\mathcal{S}'$, then the series $\sum(\bsf{1},\bsf{2})=\sum_j\Delta_{<j-1}f_1\Delta_jf_2$ coincides with $f_1\pl f_2$.
Note that $\sum(\bsf{1},\bsf{2},\bsf{3})\neq(f_1\pl f_2)\pl f_3$ in general, since $\Delta_j(f_1\pl f_2)\neq \Delta_{<j-1}f_1\Delta_jf_2$.
\end{rem}

We define the Taylor remainders for these multicomponent operators.
For any multiindices $k=(k_i)_{i=1}^d$ and $\ell=(\ell_i)_{i=1}^d$ such that $k_i\ge \ell_i$, define $\binom{k}{\ell}:=\frac{k!}{\ell!(k-\ell)!}$.

\begin{dfn}\label{dfn:Omega}
For any sequence $\bsf{n}\in\bbB^{\alpha_n}$ with $\alpha_n>0$ ($n=1,2,\dots$), recursively define
\begin{align*}
\alpha_{1\cdots n}&:=\alpha_1+\cdots+\alpha_n,\\
\{f_{1\cdots n}^j\}_{j=0}^\infty&:=(\bsf{1},\dots,\bsf{n}),\qquad
f_{1\cdots n}:=\sum_{j=0}^\infty f_{1\cdots n}^j,\\
D_{1\cdots n}^k&:=\partial^kf_{1\cdots n}-\sum_{m=1}^{n-1}
\sum_{\substack{k_1+k_2=k \\ |k_1|<\alpha_{1\cdots m} \\ |k_2|\ge\alpha_{(m+1)\cdots n}}}\binom{k}{k_1}D_{1\cdots m}^{k_1}D_{(m+1)\cdots n}^{k_2}\qquad(k\in\bbN^d),\\
T_{1\cdots n}(\cdot,x)&:=\sum_{|k|<\alpha_{1\cdots n}}\frac{(\cdot-x)^k}{k!}D_{1\cdots n}^k(x),\\
\Omega_{1\cdots n}(\cdot,x)&:=f_{1\cdots n}(\cdot)-T_{1\cdots n}(\cdot,x)-\sum_{m=1}^{n-1} T_{1\cdots m}(\cdot,x)\Omega_{(m+1)\cdots n}(\cdot,x),
\end{align*}
where $D_{(m+1)\cdots n}^k$ and $\Omega_{(m+1)\cdots n}$ are defined by the same way as $D_{1\cdots(n-m)}^k$ and $\Omega_{1\cdots(n-m)}$ respectively, by replacing $\bsf{1},\dots,\bsf{n-m}$ with $\bsf{m+1},\dots,\bsf{n}$.
In other words, $D_w^k$, $\Omega_w$, etc. are defined for any finite sequence $w=i_1i_2\cdots i_n$ of positive integers.
\end{dfn}

To make sense of the above definition, we need the following lemma.
Throughout this section, we assume that
\begin{align}\label{assump}
\alpha_{m\cdots n}=\alpha_m+\cdots+\alpha_n\notin\bbN\quad\text{for any $1\le m\le n$}.
\end{align}

\begin{lmm}\label{lem:D^kisBC}
Under the assumption \eqref{assump}, $D_{1\cdots n}^k$ is well-defined as a bounded continuous function for any $|k|<\alpha_{1\cdots n}$.
\end{lmm}

\begin{proof}
First we prove some identities.
For each $j\ge0$, we recursively define
\begin{align*}
D_{1\cdots n}^{k,j}&=\partial^kf_{1\cdots n}^j-\sum_{m=1}^{n-1}
\sum_{\substack{k_1+k_2=k \\ |k_1|<\alpha_{1\cdots m} \\ |k_2|\ge\alpha_{(m+1)\cdots n}}}\binom{k}{k_1}D_{1\cdots m}^{k_1}D_{(m+1)\cdots n}^{k_2,j},\\
C_{1\cdots n}^{k,j}&=
\partial^kf_{1\cdots n}^{<j}
-{\bf 1}_{|k|<\alpha_{1\cdots n}}D_{1\cdots n}^k
-\sum_{m=1}^{n-1}\sum_{\substack{k_1+k_2=k \\ |k_1|<\alpha_{1\cdots m}}}\binom{k}{k_1}D_{1\cdots m}^{k_1}C_{(m+1)\cdots n}^{k_2,j}.
\end{align*}
For simplicity, we write the definition of $C_{1\cdots n}^{k,j}$ as
$$
C_{1\cdots n}^{k,j}=
\partial^kf_{1\cdots n}^{<j}
-\sum_{m=1}^n\sum_{\substack{k_1+k_2=k \\ |k_1|<\alpha_{1\cdots m}}}\binom{k}{k_1}D_{1\cdots m}^{k_1}C_{(m+1)\cdots n}^{k_2,j}
$$
in short.
Then we can derive
\begin{align}\label{eq:WDofD}
\begin{aligned}
&D_{1\cdots n}^{k,j}
=\sum_{k=k_1+k_2}\binom{k}{k_1}C_{1\cdots(n-1)}^{k_1,j}\partial^{k_2}f_n^{j}\\
&+\sum_{r=2}^n(-1)^r\sum_{w_1,w_2,\dots,w_r}
\sum_{\substack{k_1+k_2+\cdots+k_r=k \\ |k_1|<\alpha_{w_1},\dots,\, |k_r|<\alpha_{w_r}}}
\frac{k!}{k_1!k_2!\cdots k_r!}D_{w_1}^{k_1}D_{w_2}^{k_2}\cdots D_{w_{r-1}}^{k_{r-1}}D_{w_r}^{k_r,j},
\end{aligned}
\end{align}
where $w_1,w_2,\dots,w_r$ runs over all partitions of the sequence $1\cdots n$,
that is, $w_i$ is a sequence from $n_{i-1}+1$ to $n_i$ for some $0=n_0< n_1<n_2<\cdots<n_r=n$.
To prove \eqref{eq:WDofD}, we transform the definition of $D_{1\cdots n}^{k,j}$ into
\begin{align}\label{eq:Ddecomp1}
\begin{aligned}
D_{1\cdots n}^{k,j}
&=\sum_{k_1+k_2=k}\binom{k}{k_1}\partial^{k_1}f_{1\cdots (n-1)}^{<j}\partial^{k_2}f_n^j
-\sum_{m=1}^{n-1}\sum_{\substack{k_1+k_2=k \\ |k_1|<\alpha_{1\cdots m}}}\binom{k}{k_1}D_{1\cdots m}^{k_1}D_{(m+1)\cdots n}^{k_2,j}\\
&\quad+\sum_{m=1}^{n-1}\sum_{\substack{k_1+k_2=k \\ |k_1|<\alpha_{1\cdots m},\, |k_2|<\alpha_{(m+1)\cdots n}}}\binom{k}{k_1}D_{1\cdots m}^{k_1}D_{(m+1)\cdots n}^{k_2,j}.
\end{aligned}
\end{align}
Note that we used Leibniz rule for $f_{1\cdots n}^j=f_{1\cdots(n-1)}^{<j}f_n^j$.
Then by using the decomposition \eqref{eq:WDofD} for $D_{(m+1)\cdots n}^{k_2,j}$ of the second term in the right-hand side of \eqref{eq:Ddecomp1} over $1\le m\le n-2$
(and $D_n^{k_2,j}=\partial^{k_2}f_n^j$ for $m=n-1$), we can obtain the same formula \eqref{eq:WDofD} for $D_{1\cdots n}^{k,j}$.

From \eqref{eq:WDofD}, we can recursively show that $\|D_{1\cdots n}^{k,j}\|_{L^\infty}\lesssim 2^{-j\theta_{n,k}}$ with some $\theta_{n,k}>0$ for any $|k|<\alpha_{1\cdots n}$, hence $D_{1\cdots n}^k:=\sum_{j=0}^\infty D_{1\cdots n}^{k,j}$ converges in the supremum norm.
Such an estimate is obtained from $\|\partial^kf_n^j\|_{L^\infty}\lesssim 2^{-j(\alpha_n-|k|)}\|\bsf{n}\|_{\alpha_n}$ and the next lemma.
\end{proof}

\begin{lmm}\label{lmm3}
$\|C_{1\cdots n}^{k,j}\|_{L^\infty}\lesssim2^{-j(\alpha_{1\cdots n}-|k|)}\|\bsf{1}\|_{\alpha_1}\cdots\|\bsf{n}\|_{\alpha_n}$.
\end{lmm}

\begin{proof}
We assume that the estimate holds for $C_{1\cdots m}^{k,j}$ with $m<n$.
Then it is sufficient to show the formula
\begin{align}\label{eq:explicitC}
C_{1\cdots n}^{k,j}
=
\left\{
\begin{aligned}
&-\sum_{k=k_1+k_2}\binom{k}{k_1}\sum_{i\ge j}C_{1\cdots(n-1)}^{k_1,i}\partial^{k_2}f_n^i,\qquad&|k|<\alpha_{1\cdots n},\\
&\sum_{k=k_1+k_2}\binom{k}{k_1}\sum_{i<j}C_{1\cdots(n-1)}^{k_1,i}\partial^{k_2}f_n^i,\qquad&|k|\ge\alpha_{1\cdots n}.
\end{aligned}
\right.
\end{align}
From this formula and the assumption \eqref{assump}, we can recursively prove the required estimate.

We prove \eqref{eq:explicitC}.
If $|k|\ge\alpha_{1\cdots n}$, since ${\bf 1}_{|k|<\alpha_{1\cdots n}}D_{1\cdots n}^k=0$ we obtain
\begin{align*}
C_{1\cdots n}^{k,j}
&=\partial^kf_{1\cdots n}^{<j}-\sum_{m=1}^{n-1}\sum_{\substack{k_1+k_2=k \\ |k_1|<\alpha_{1\cdots m}}}\binom{k}{k_1}D_{1\cdots m}^{k_1}C_{(m+1)\cdots n}^{k_2,j}\\
&=\sum_{i<j}\Big\{\partial^k(f_{1\cdots (n-1)}^{<i}f_n^i)-\sum_{m=1}^{n-1}\sum_{\substack{k_1+k_2+k_3=k \\ |k_1|<\alpha_{1\cdots m}}}\frac{k!}{k_1!k_2!k_3!}D_{1\cdots m}^{k_1}C_{(m+1)\cdots(n-1)}^{k_2,i}\partial^{k_3}f_n^i\Big\}\\
&=\sum_{i<j}\sum_{k=k_1+k_2}\binom{k}{k_1}C_{1\cdots(n-1)}^{k_1,i}\partial^{k_2}f_n^i.
\end{align*}
In the second equality, we used the latter part of \eqref{eq:explicitC} for $C_{(m+1)\cdots n}^{k_2,j}$, since the conditions on $k$ and $k_1$ imply $|k_2|>\alpha_{(m+1)\cdots n}$.
To consider the case $|k|<\alpha_{1\cdots n}$, we decompose
\begin{align*}
C_{1\cdots n}^{k,j}
&=\partial^kf_{1\cdots n}^{<j}-D_{1\cdots n}^k-\sum_{m=1}^{n-1}\sum_{\substack{k_1+k_2=k \\ |k_1|<\alpha_{1\cdots m}}}\binom{k}{k_1}D_{1\cdots m}^{k_1}C_{(m+1)\cdots n}^{k_2,j}\\
&=\partial^kf_{1\cdots n}^{<j}-\partial^kf_{1\cdots n}
+\sum_{m=1}^{n-1}\sum_{\substack{k_1+k_2=k \\ |k_1|<\alpha_{1\cdots m} \\ |k_2|\ge\alpha_{(m+1)\cdots n}}}
\binom{k}{k_1}D_{1\cdots m}^{k_1}D_{(m+1)\cdots n}^{k_2}\\
&\quad-\sum_{m=1}^{n-1}\sum_{\substack{k_1+k_2=k \\ |k_1|<\alpha_{1\cdots m} \\ |k_2|<\alpha_{(m+1)\cdots n}}}\binom{k}{k_1}D_{1\cdots m}^{k_1}C_{(m+1)\cdots n}^{k_2,j}\\
&\quad-\sum_{m=1}^{n-1}\sum_{\substack{k_1+k_2=k \\ |k_1|<\alpha_{1\cdots m} \\ |k_2|\ge\alpha_{(m+1)\cdots n}}}\binom{k}{k_1}D_{1\cdots m}^{k_1}C_{(m+1)\cdots n}^{k_2,j}.
\end{align*}
We write five terms as $(A)+(B)+\sum_{m=1}^{n-1}(C)_m+\sum_{m=1}^{n-1}(D)_m+\sum_{m=1}^{n-1}(E)_m$ in order.
First, $(A)+(B)=-\partial^kf_{1\cdots n}^{\ge j}$ in distributional sense.
Here and in what follows, we write $a^{\ge j}:=\sum_{i\ge j}a^i$ for any summable sequences $\{a^i\}_{i\in\bbN}$.
For $(D)$, by using the former part of \eqref{eq:explicitC} for $C_{(m+1)\cdots n}^{k_2,j}$, we have
$$
(D)_m=\sum_{\substack{k_1+k_2+k_3=k \\ |k_1|<\alpha_{1\cdots m} \\ |k_2+k_3|<\alpha_{(m+1)\cdots n}}}
\frac{k!}{k_1!k_2!k_3!}
D_{1\cdots m}^{k_1}\big(C_{(m+1)\cdots(n-1)}^{k_2,(\cdot)}\partial^{k_3}f_n^{(\cdot)}\big)^{\ge j}.
$$
For $(C)$, by the formula \eqref{eq:WDofD}, we have
$$
(C)_m
=\sum_{\substack{k_1+k_2+k_3=k \\ |k_1|<\alpha_{1\cdots m} \\ |k_2+k_3|\ge\alpha_{(m+1)\cdots n}}}
\frac{k!}{k_1!k_2!k_3!}D_{1\cdots m}^{k_1}
\sum_{j=0}^\infty C_{(m+1)\cdots (n-1)}^{k_2,j}\partial^{k_3}f_n^j.
$$
Note that the second term of \eqref{eq:WDofD} does not appear in this case, since $|k_2|\ge\alpha_{(m+1)\cdots n}$.
Hence by the latter part of \eqref{eq:explicitC}, we have
\begin{align*}
(C)_m+(E)_m
&=\sum_{\substack{k_1+k_2+k_3=k \\ |k_1|<\alpha_{1\cdots m} \\ |k_2+k_3|\ge\alpha_{(m+1)\cdots n}}}
\frac{k!}{k_1!k_2!k_3!}
D_{1\cdots m}^{k_1}\big(C_{(m+1)\cdots(n-1)}^{k_2,(\cdot)}\partial^{k_3}f_n^{(\cdot)}\big)^{\ge j}.
\end{align*}
Consequently, we can write $C_{1\cdots n}^{k,j}=-F^{\ge j}$ with
\begin{align*}
F^j&=\partial^kf_{1\cdots n}^j-\sum_{m=1}^{n-1}\sum_{\substack{k_1+k_2+k_3=k \\ |k_1|<\alpha_{1\cdots m} }}
\frac{k!}{k_1!k_2!k_3!}
D_{1\cdots m}^{k_1}C_{(m+1)\cdots(n-1)}^{k_2,j}\partial^{k_3}f_n^j\\
&=\sum_{k=k_1+k_2}\binom{k}{k_1}C_{1\cdots(n-1)}^{k_1,j}\partial^{k_2}f_n^j,
\end{align*}
which completes the proof of \eqref{eq:explicitC}.
\end{proof}

We return to Definition \ref{dfn:Omega}.
The two-parameter function $\Omega_{1\cdots n}$ turns out to have a role of Taylor remainder.

\begin{thm}\label{thm:abstractparaproduct}
Under the assumption \eqref{assump}, for any $n\in\bbN$,
$$
\sup_{x,y\in\bbR^d}\frac{|\Omega_{1\cdots n}(y,x)|}{|y-x|^{\alpha_{1\cdots n}}}
\lesssim\|\bsf{1}\|_{\alpha_1}\cdots\|\bsf{n}\|_{\alpha_n}.
$$
\end{thm}

The main part of the proof is the decomposition of $\Omega_{1\cdots n}$ into good pieces.
For any $\theta\in\bbR$, define 
\begin{align*}
T_{1\cdots n}^{\theta,j}(\cdot,x)&=\sum_{|k|<\theta}\frac{(\cdot-x)^k}{k!}D_{1\cdots n}^{k,j}(x),\\
\Omega_{1\cdots n}^{\theta,j}(\cdot,x)&=f_{1\cdots n}^j(x)-T_{1\cdots n}^{\theta,j}(\cdot,x)-\sum_{m=1}^{n-1}T_{1\cdots m}(\cdot,x)\Omega_{(m+1)\cdots n}^{\alpha_{(m+1)\cdots n},j}(\cdot,x).
\end{align*}
Then $\Omega_{1\cdots n}=\sum_{j=0}^\infty\Omega_{1\cdots n}^{\alpha_{1\cdots n},j}$.
Note that $\Omega_{1\cdots n}^{\theta,j}=\Omega_{1\cdots n}^{\alpha_{1\cdots n},j}$ for any $\theta$ with $[\alpha_{1\cdots n}]<\theta\le[\alpha_{1\cdots n}]+1$.

\begin{lmm}\label{lmm2}
Let $n\ge2$.
For any $\theta>[\alpha_{1\cdots n}]$ and $j$,
\begin{align}\label{eq:lmm2}
\Omega_{1\cdots n}^{\theta,j}(\cdot,x)
=\Omega_{1\cdots(n-1)}^{\theta,<j}(\cdot,x)f_n^j(\cdot)
+\sum_{|k|<\theta}\frac{(\cdot-x)^k}{k!}C_{1\cdots(n-1)}^{k,j}(x)\Omega_n^{\theta-|k|,j}(\cdot,x).
\end{align}
\end{lmm}

\begin{proof}
Recall that we defined the notation
$$
\Omega^\theta(f)(\cdot,x):=f(\cdot)-\sum_{|k|<\theta}\frac{(\cdot-x)^k}{k!}\partial^kf(x)
$$
in \eqref{1:def:Omegaalpha}.
Before proving \eqref{eq:lmm2}, we show the intermediate formula
\begin{align}\label{eq:lmm1}
\Omega_{1\cdots n}^{\theta,j}(\cdot,x)
=\Omega^\theta(f_{1\cdots n}^j)(\cdot,x)
-\sum_{m=1}^{n-1}\sum_{|k|<\alpha_{1\cdots m}}\frac{(\cdot-x)^k}{k!}D_{1\cdots m}^k(x)\Omega_{(m+1)\cdots n}^{\theta-|k|,j}(\cdot,x)
\end{align}
for any $\theta>[\alpha_{1\cdots n}]$. This formula is used in the next step.
By definition of $\Omega^{\theta,j}$, the identity
$$
\Omega_{(m+1)\cdots n}^{\theta-|k|,j}(\cdot,x)=\Omega_{(m+1)\cdots n}^{\alpha_{(m+1)\cdots n},j}(\cdot,x)-\sum_{\alpha_{(m+1)\cdots n}\le|\ell|<\theta-|k|}\frac{(\cdot-x)^\ell}{\ell!}D_{(m+1)\cdots n}^{\ell,j}(x)
$$
holds since $\theta-|k|>[\alpha_{(m+1)\cdots n}]$.
By inserting it into the right-hand side of \eqref{eq:lmm1} and using the definition of $D_{1\cdots n}^{k,j}$, we obtain the identity \eqref{eq:lmm1}.

Next we prove \eqref{eq:lmm2} from \eqref{eq:lmm1} by an induction for $n$.
In general, an elementary computation with Leibniz rule yields that for any smooth $f$ and $g$,
\begin{align*}
\Omega^\theta(fg)(\cdot,x)
=\Omega^\theta(f)(\cdot,x)g(\cdot)
+\sum_{|k|<\theta}\frac{(\cdot-x)^k}{k!}\partial^kf(x)\Omega^{\theta-|k|}(g)(\cdot,x).
\end{align*}
We write the right-hand side of \eqref{eq:lmm1} as $({\rm I})-\sum_{m=1}^{n-1}({\rm I\!I})_m$ in short.
For $({\rm I})$, since $f_{1\cdots n}^j=f_{1\cdots(n-1)}^{<j}f_n^j$,
\begin{align*}
({\rm I})
&=\Omega^\theta(f_{1\cdots(n-1)}^{<j})(\cdot,x)f_n^j(\cdot)
+\sum_{|k|<\theta}\frac{(\cdot-x)^k}{k!}\partial^kf_{1\cdots(n-1)}^{<j}(x)\Omega_n^{\theta-|k|,j}(\cdot,x)\\
&=:({\rm I})_1+({\rm I})_2.
\end{align*}
For $({\rm I\!I})_m$ with $m\le n-2$, by assuming \eqref{eq:lmm2} for $\Omega_{(m+1)\cdots n}^{\theta-|k|,j}$ by induction,
\begin{align*}
({\rm I\!I})_m
&=\sum_{|k|<\alpha_{1\cdots m}}\frac{(\cdot-x)^k}{k!}D_{1\cdots m}^k(x)\Omega_{(m+1)\cdots(n-1)}^{\theta-|k|,<j}(\cdot,x)f_n^j(\cdot)\\
&\quad+\sum_{|k|<\alpha_{1\cdots m},\, |\ell|<\theta-|k|}
\frac{(\cdot-x)^{k+\ell}}{k!\ell!}D_{1\cdots m}^k(x)C_{(m+1)\cdots(n-1)}^{\ell,j}(x)\Omega_n^{\theta-|k+\ell|,j}(\cdot,x)\\
&=:({\rm I\!I})_{m,1}+({\rm I\!I})_{m,2}.
\end{align*}
Then $({\rm I})_1-\sum_{m=1}^{n-2}({\rm I\!I})_{m,1}$ and $({\rm I})_2-({\rm I\!I})_{n-1}-\sum_{m=1}^{n-2}({\rm I\!I})_{m,2}$ respectively produces the first and second terms of the right-hand side of \eqref{eq:lmm2}.
\end{proof}

\begin{proof}[Proof of Theorem \ref{thm:abstractparaproduct}]
We inductively show that
\begin{align}\label{ineq:thetadependent}
|\Omega_{1\cdots n}^{\theta,j}(y,x)|\lesssim2^{-j(\alpha_{1\cdots n}-\theta)}|y-x|^\theta
\end{align}
for any non-integer $\theta>\alpha_{1\cdots(n-1)}\vee[\alpha_{1\cdots n}]$.
For $n=1$ and non-integer $\theta>0$, two Taylor expansions up to order $[\theta]$ and $[\theta]-1$ yield the estimates
\begin{align*}
|\Omega_1^{\theta,j}(y,x)|=|\Omega^\theta(f^j)(y,x)|
&\lesssim\sum_{|k|=[\theta]+1}\|\partial^kf^j\|_{L^\infty(\mathbb{R}^d)}|y-x|^{|k|}\\
&\lesssim2^{-j(\alpha_1-[\theta]-1)}|y-x|^{[\theta]+1}
\end{align*}
and
\begin{align*}
|\Omega_1^{\theta,j}(y,x)|=\Big|\Omega^{\theta-1}(f^j)(y,x)-\sum_{|k|=[\theta]}(\cdots)\Big|
\lesssim2^{-j(\alpha_1-[\theta])}|y-x|^{[\theta]}.
\end{align*}
By an interpolation, we have 
\begin{align}\label{2:ineq:Omega1Taylor}
|\Omega_1^{\theta,j}(y,x)|\lesssim2^{-j(\alpha_1-\theta)}|y-x|^\theta
\end{align}
For $n\ge2$, we use Lemma \ref{lmm2}.
The second term of the right-hand side of \eqref{eq:lmm2} is bounded by
$$
\sum_{|k|<\theta}|y-x|^{|k|}2^{-j(\alpha_{1\cdots(n-1)}-|k|)}2^{-j(\alpha_n-\theta+|k|)}|y-x|^{\theta-|k|}
\lesssim 2^{-j(\alpha_{1\cdots n}-\theta)}|y-x|^\theta,
$$
by Lemma \ref{lmm3} and the estimate \eqref{2:ineq:Omega1Taylor} of the Taylor remainder for $f_n^j$.
The first term is also bounded by
$$
\sum_{i<j}2^{-i(\alpha_{1\cdots (n-1)}-\theta)}|y-x|^\theta\cdot2^{-j\alpha_n}\lesssim2^{-j(\alpha_{1\cdots n}-\theta)}|y-x|^\theta,
$$
by an inductive assumption \eqref{ineq:thetadependent} on $\Omega_{1\cdots(n-1)}^{\theta,j}$, and since $\theta>\alpha_{1\cdots(n-1)}$.

Since $\Omega_{1\cdots n}^{\theta,j}=\Omega_{1\cdots n}^{\alpha_{1\cdots n},j}$ for any $\theta$ in a small neighborhood of $\alpha_{1\cdots n}$, the next lemma yields Theorem \ref{thm:abstractparaproduct}.
\end{proof}

\begin{lmm}[{\cite[Lemma 3.7]{Hos20}, \cite[Lemma 2.5]{Hos21}}]
Let $\{\Omega^j\}_{j\ge0}$ be a family of functions on $\bbR^d\times\bbR^d$ such that, for some $C>0$ and $\alpha>0$, the bound
$$
|\Omega^j(y,x)|\le C2^{j(\theta-\alpha)}|y-x|^\theta,\qquad x,y\in\bbR^d
$$
holds for any $\theta$ in a neighborhood of $\alpha$. Then $\Omega(y,x)=\sum_{j=0}^\infty\Omega^j(y,x)$ absolutely converges and
$$
|\Omega(y,x)|\lesssim C|y-x|^\alpha,\qquad x,y\in\bbR^d.
$$
\end{lmm}


\section{Proof of Theorem \ref{mainthm}}\label{section:pointwise}

In this section, we prove Theorem \ref{mainthm}.
The two-component case (1) is a particular one of Theorem \ref{thm:abstractparaproduct} with $n=2$, $(\alpha_1,\alpha_2)=(\alpha,\beta)$, and
$$
\bsf{1}=\{f_1^j=\Delta_{j-1}f\}_{j=0}^\infty\in\bbB^\alpha,\qquad
\bsf{2}=\{f_2^j=\Delta_{j-1}g\}_{j=0}^\infty\in\bbB^\beta.
$$
We define $(\bsf{1},\bsf{2})=\{f_1^{<j-1}f_2^j\}_{j=0}^\infty$ instead of \eqref{eq:abstractparaproduct2} to obtain $f_{12}=\sum(\bsf{1},\bsf{2})=f\pl g$
(see Remark \ref{rem:<j-N}). Then we have
\begin{align*}
\Omega_{12}(\cdot,x)=f\pl g(\cdot)-\sum_{|k|<\alpha+\beta}\frac{(\cdot-x)^k}{k!}D_{12}^k(x)-\sum_{|k|<\alpha}\frac{(\cdot-x)^k}{k!}\partial^kf(x)\Omega^\beta(g)(\cdot,x)
\end{align*}
with a bounded continuous function $D_{12}^k=D^k(f,g)$ as in \eqref{eq:D(f,g)}.

From now on, we focus on the three-component case (2).

\begin{lmm}\label{lmm:R}
Under the assumptions of Theorem \ref{mainthm},
$$
\Big\{R^j:=\Delta_{j-1}(f\pl g)-\sum_{|k|<\alpha}\Delta_{<j-2}(\partial^kf)\Delta_{j-1}^kg\Big\}_{j=0}^\infty\in\bbB^{\alpha+\beta},
$$
where $\Delta_{j-1}^kg$ is defined by
$$
\Delta_{j-1}^kg(x)=\int_{\bbR^d}(\mathcal{F}^{-1}\rho_{j-1})(x-y)\frac{(y-x)^k}{k!}g(y)dy.
$$
\end{lmm}

\begin{proof}
It is sufficient to consider $j\ge0$.
We integrate the both sides of \eqref{mainthm1} multiplied by $\mathcal{F}^{-1}(\rho_{j-1})(x-y)$ over $y$.
Since $\int\mathcal{F}^{-1}(\rho_{j-1})(x)x^kdx=0$ for any $k\in\bbN^d$, only the terms
\begin{align*}
\Delta_{j-1}(f\pl g)-\sum_{|k|<\alpha}\partial^kf\Delta_{j-1}^kg=O(2^{-j(\alpha+\beta)})
\end{align*}
remain.
Since $\Delta_i\Delta_j^kg=\Delta_j^k\Delta_ig$ vanishes if $|i-j|\ge2$, we can see 
$$
|\Delta_{j-1}^kg(x)|\le\int\Big|\mathcal{F}^{-1}(\rho_{j-1})(x)\frac{x^k}{k!}\Big|dx\sum_{|i-j|\le1}\|\Delta_{i-1}g\|_{L^\infty}\lesssim2^{-j(\beta+|k|)}
$$
by the scaling property of $\mathcal{F}^{-1}\rho_j$.
Since
$$
\|\Delta_{\ge j-2}(\partial^kf)\Delta_{j-1}^kg\|_{L^\infty}\lesssim2^{-j(\alpha-|k|)}2^{-j(\beta+|k|)}=2^{-j(\alpha+\beta)}
$$
for any $|k|<\alpha$, we have the required estimate for $R^j$.
\end{proof}

\begin{proof}[Proof of Theorem \ref{mainthm}]
Let
\begin{align*}
&\bsf{1}^{\, (k)}=\{\Delta_{j-1}(\partial^kf)\}_{j=0}^\infty\in\bbB^{\alpha-|k|},\qquad
\bsf{2}^{\, (k)}=\{\Delta_{j-1}^kg\}_{j=0}^\infty\in\bbB^{\beta+|k|},\\
&\bsf{3}=\{R^j\}_{j=0}^\infty\in\bbB^{\alpha+\beta},\qquad
\bsf{4}=\{\Delta_{j-1}h\}_{j=0}^\infty\in\bbB^\gamma.
\end{align*}
Then by Lemma \ref{lmm:R}, we can write
$$
f\pl g=\sum_{|k|<\alpha}f_{12}^{(k)}+f_3,\qquad
(f\pl g)\pl h=\sum_{|k|<\alpha}f_{124}^{(k)}+f_{34},
$$
where $f_{12}^{(k)}=\sum(\bsf{1}^{\, (k)},\bsf{2}^{\, (k)})$, $f_3=\sum\bsf{3}$, $f_{124}^{(k)}=\sum(\bsf{1}^{\, (k)},\bsf{2}^{\, (k)},\bsf{4})$, and $f_{34}=\sum(\bsf{3},\bsf{4})$.
Here we use the definition $(\boldsymbol{f},\boldsymbol{g})=\{f^{<j-1}g^j\}$ instead of \eqref{eq:abstractparaproduct2}.
By Theorem \ref{thm:abstractparaproduct}, each of $f_{124}^{(k)}$ and $f_{34}$ has the estimate of the forms
\begin{align*}
\Omega_{124}^{(k)}&=f_{124}^{(k)}-T_{124}^{(k)}-T_1^{(k)}\Omega_{24}^{(k)}-T_{12}^{(k)}\Omega_4=O(|\cdot-x|^{\alpha+\beta+\gamma})\qquad(|k|<\alpha),\\
\Omega_{34}&=f_{34}-T_{34}-T_3\Omega_4=O(|\cdot-x|^{\alpha+\beta+\gamma}).
\end{align*}
We can conclude Theorem \ref{mainthm} by reorganizing the sum $\sum_{|k|<\alpha}\Omega_{124}^{(k)}+\Omega_{34}$.
The sum $\sum_{|k|<\alpha}T_{124}^{(k)}+T_{34}$ is reorganized as the polynomial part $\sum_{|\ell|<\alpha+\beta+\gamma}\frac{(\cdot-x)^\ell}{\ell!}D^\ell(f,g,h)(x)$.
Next we consider the terms containing $\Omega_4$. Note that $T_{12}^{(k)}=\sum_{|\ell|<\alpha+\beta}\frac{(\cdot-x)^\ell}{\ell!}D_{12}^{(k),\ell}$, where
\begin{align*}
D_{12}^{(k),\ell}=\partial^\ell f_{12}^{(k)}-\sum_{\substack{\ell_1+\ell_2=\ell \\ |\ell_1|<\alpha-|k|,\ |\ell_2|\ge\beta+|k|}}\binom{\ell}{\ell_1}\partial^{\ell_1}f_1^{(k)}\partial^{\ell_2}f_2^{(k)}.
\end{align*}
Since $f_2^{(k)}=\sum_{j=-1}^\infty\Delta_j^kg=\int\delta(\cdot-y)\, (y-\cdot)^kg(y)dy$ vanishes if $k\neq0$,
we have
$\sum_{|k|<\alpha}T_{12}^{(k)}+T_3=\sum_{|\ell|<\alpha+\beta}\frac{(\cdot-x)^\ell}{\ell!}D^\ell(x)$ with
\begin{align*}
D^\ell
&=\sum_{|k|<\alpha}D_{12}^{(k),\ell}+D_3^\ell\\
&=\sum_{|k|<\alpha}\partial^\ell f_{12}^{(k)}-\sum_{\substack{\ell_1+\ell_2=\ell \\ |\ell_1|<\alpha,\ |\ell_2|\ge\beta}}\binom{\ell}{\ell_1}\partial^{\ell_1}f_1^{(0)}\partial^{\ell_2}f_2^{(0)}+\partial^\ell f_3\\
&=\partial^\ell(f\pl g)-\sum_{\substack{\ell_1+\ell_2=\ell \\ |\ell_1|<\alpha,\ |\ell_2|\ge\beta}}\binom{\ell}{\ell_1}\partial^{\ell_1}f\partial^{\ell_2}g\\
&=D^\ell(f,g).
\end{align*}
Hence $(\sum_{|k|<\alpha}T_{12}^k+T_3)\Omega_4=\sum_{|\ell|<\alpha+\beta}\frac{(\cdot-x)^\ell}{\ell!}D^\ell(f,g)(x)\Omega^\gamma(h)$.
Finally we consider the term $T_1^{(k)}\Omega_{24}^{(k)}$.
For $k=0$, we have
$$
T_1^{(0)}\Omega_{24}^{(0)}
=\sum_{|\ell|<\alpha}\frac{(\cdot-x)^\ell}{\ell!}\partial^\ell f(x)\Omega^{\beta,\gamma}(g,h)(\cdot,x).
$$
For $k\neq0$, since $f_2^{(k)}=0$,
$$
\Omega_{24}^{(k)}
=f_{24}^{(k)}-\sum_{|\ell|<\beta+\gamma+|k|}\frac{(\cdot-x)^\ell}{\ell!}\partial^\ell f_{24}^{(k)},
$$
where $f_{24}^{(k)}=\sum_{j}\Delta_{<j-1}^kg\Delta_jh=:g\mpl{k}h$. This is the third term of the right-hand side of \eqref{mainthm2}.
\end{proof}

\section*{Acknowledgements}

The author thanks anonymous referee for helpful comments that improved the quality of the paper.



\begin{thebibliography}{99}
%
%

\bibitem{BB19}
Bailleul,~I. and Bernicot,~F.,
High order paracontrolled calculus, 
\textit{Forum Math. Sigma}, \textbf{7} (2019), e44, 94 pp.

\bibitem{BCD11}
Bahouri,~H., Chemin,~J.-Y., and Danchin,~R.,
\textit{Fourier Analysis and Nonlinear Partial Differential Equations},
Springer, 2011.

\bibitem{BH21a}
Bailleul,~I. and Hoshino,~M.,
Paracontrolled calculus and regularity structures I,
\textit{J. Math. Soc. Japan}, \textbf{73} (2021), 553--595.

\bibitem{BH21b}
Bailleul,~I. and Hoshino,~M.,
Paracontrolled calculus and regularity structures II,
\textit{J. \'{E}c. polytech. Math.}, \textbf{8} (2021), 1275--1328.

\bibitem{GIP15}
Gubinelli,~M, Imkeller,~P., and Perkowski,~N.,
Paracontrolled distributions and singular PDEs,
\textit{Forum Math. Pi}, \textbf{3} (2015), e6, 75pp.

\bibitem{Hai14}
Hairer,~M.,
A theory of regularity structures,
\textit{Invent. Math.}, \textbf{198} (2014), 269--504.

\bibitem{Hos20}
Hoshino,~M.,
Commutator estimates from a viewpoint of regularity structures,
\textit{RIMS K\^oky\^uroku Bessatsu}, B79 (2020), 179--197,
\texttt{arXiv:1903.00623}.
 
\bibitem{Hos21}
Hoshino,~M.,
Iterated paraproducts and iterated commutator estimates in Besov spaces,
\textit{Advanced Studies in Pure Mathematics}, \textbf{87} (2021), 239-259 (Proceedings of ``The 12th Mathematical Society of Japan Seasonal Institute, Stochastic Analysis, Random Fields and Integrable Probability"),
\texttt{arXiv:2001.07414}.

\bibitem{MP20}
Martin,~J. and Perkowski,~N.,
A Littlewood-Paley description of modelled distributions,
\textit{J. Funct. Anal.}, \textbf{279} (2020), 108634, 22 pp.

%
%
%
%
%
%
%
%
%
%
%
%
%
%
%
%
%
%
%

\end{thebibliography}
\end{document}